\documentclass[12pt]{amsart}
\usepackage{amsfonts}
\usepackage{amsmath}
\usepackage{amssymb}
\usepackage{graphicx}
\usepackage{enumerate}
\usepackage{multicol}
\usepackage{mathrsfs}
\usepackage[usenames,dvipsnames]{pstricks}
\usepackage{epsfig}
\usepackage{pst-grad} 
\usepackage{pst-plot} 
\usepackage[hmargin=3cm,vmargin=2.5cm]{geometry}

\newtheorem{theorem}{Theorem}
\newtheorem{corollary}{Corollary}
\newtheorem{lemma}{Lemma}
\newtheorem{definition}{Definition}
\newtheorem{proposition}{Proposition}
\theoremstyle{remark}
\newtheorem{remark}{Remark}
\newtheorem{example}{Example}

\newcommand{\Z}{{\mathbb{Z}}}

\newcommand{\R}{{\mathbb{R}}}

\newcommand{\N}{{\mathbb{N}}}

\DeclareMathOperator{\rank}{rank}

\newcommand{\inv}{{^{-1}}}
\newcommand{\half}{\frac{1}{2}}
\newcommand{\quarter}{\frac{1}{4}}

\newcommand{\ts}{\mathrm{T^S}}

\newcommand{\Ra}{\Rightarrow}
\newcommand{\LRa}{\Leftrightarrow}

\title{On the Alexandrov Topology of sub-Lorentzian Manifolds}
\author[Irina Markina, Stephan Wojtowytsch]{Irina Markina \\ Stephan Wojtowytsch}

\address{Department of Mathematics, University of Bergen, Norway.}
\email{irina.markina@math.uib.no}

\address{Fakultaet fuer Mathematik und Informatik, Universitaet Heidelberg, Germany}
\email{wojtowytsch@stud.uni-heidelberg.de}

\thanks{The work of both authors are partially supported by NFR-FRINAT grants \#204726/V30 and \#213440/BG}



\begin{document}
\maketitle

\begin{abstract} 
In the present work we show that in contrast to sub-Riemannian geometry, in sub-Lorentzian geometry the manifold topology, the topology generated by an analogue of the Riemannian distance function and the Alexandrov topology based on causal relations, are not equivalent in general and may possess a variety of relations. We also show that `opened causal relations' are more well-behaved in sub-Lorentzian settings.
\end{abstract}

\section{Introduction}
Recall that a SemiRiemannian (or PseudoRiemannian) manifold is a $C^{\infty}$-smooth manifold $M$ equipped with a non-degenerate symmetric tensor $g$. The tensor defines a scalar product on the tangent space at each point. The quadratic form corresponding to the scalar product can have different numbers of negative eigenvalues. If the quadratic form is positively definite everywhere, the manifold is usually called Riemannian. The special case of one negative eigenvalue received the name the Lorentzian manifold.

Let us assume that a smooth subbundle $D$ of the tangent bundle $TM$ is given, on which we shall later impose certain non-integrability conditions. Suppose also that a non-degenerate symmetric tensor $g$ defines a scalar product $g_D$ on planes $D_q\subset T_qM$. Then the triplet $(M,D,g_D)$ is called a sub-SemiRiemannian manifold. If $g_D$ is positive definite everywhere, then  $(M,D,g_D)$ is called a sub-Riemannian manifold. Those manifolds are an active area of research, see, for instance~\cite{AgrSach,CDPT,LS,M,Str}. In the case of exactly one negative eigenvalue the manifold is called sub-Lorentzian. This setting has been considered in~\cite{CMV, G,G1,G2,GV,KM1,KM2,KM3}.

Sub-SemiRiemannian manifolds are an abstract setting for mechanical systems with non-holonomic constraints, linear and affine control systems, the motion of particles in magnetic fields, Cauchy-Riemann geometry and other subjects from pure and applied mathematics. 

The main goal of the present work is to study the relations between the given manifold topolog and the Alexandrov and time-separation topologies defined by causality properties in the presence of a sub-Lorentzian metric. The main aim is to compare these topologies for Lorentzian and sub-Lorentzian manifolds. 

Recall that on Riemannian manifolds the original topology of the manifold and metric topology defined by the Riemannian distance function are equivalent.  Due to the Ball-Box theorem we observe that in sub-Riemannian geometry the metric topology induced by the sub-Riemannian distance function and the original  manifold topology are equivalent, too. 
\begin{theorem}[Ball-Box Theorem]
Let $(M,D,g_D)$ be a sub-Riemannian manifold. Then for every point $p\in M$ there exist coordinates $(U,x)$ around $p$ and constants $c,C>0$ such that the sub-Riemannian distance function $d_{sR}$ defined by
$$
\begin{array}{llll}
&d_{sR}(p,q)=\inf\Big\{&\text{length of absolutely continuous curves}\ \ \gamma\colon [0,1]\to M,\ \gamma(0)=p,
\\
& &\ \ \gamma(1)=q,\ \ \dot\gamma(t)\in D_{\gamma(t)}\ \ \text{for almost all}\  \  t\Big\},
\end{array}
$$
can be estimated by $
c\:\sum_{i=1}^n|x^i|^\frac{1}{w_i}\leq d_{sR}(p,q) \leq C\:\sum_{i=1}^n|x^i|^\frac{1}{w_i},\quad x\in U$, 
where the constants $w_i\in\N$ are determined by the non-integrability properties of $D$ at the point $p$. 
\end{theorem}

In the Lorentzian and sub-Lorentzian cases we cannot obtain a metric distance function from a given indefinite scalar product. The closest analogue is the time separation function,  which behaves quite differently from metric distances in a number of aspects.

The other specific feature of Lorentzian manifolds is a causal structure. In a sense, causality theory is the natural replacement for the metric geometry of Riemannian manifolds in the Lorentzian case. From causal relations one obtains a new topology, called the {\it Alexandrov topology}. It is known that in Lorentzian manifolds the Alexandrov topology can be obtained also from the time separation function. We showed that for sub-Lorentzian manifolds this is not generally true anymore and that the time separation function defines a new {\it time separation topology}, that can be thought of as a different extension of the Alexandrov topology. Our main interest is the study these two topologies, their similarities and differences from Lorentzian Alexandrov topologies. We proved that the Alexandrov and time separation topologies on a sub-Lorentzian manifold do not generally coincide neither with each other nor with the manifold topology. 

We introduce a third extension of the Alexandrov topology to sub-Lorentzian manifolds, in which we force the sets from the Alexandrov topology to be open with respect to the manifold topology. They will correspond to what we call {\it opened causal relations}. The opened causal relations provide a useful tool both for our purpose of comparison of topologies and for the generalization of the causal hierarchy of space-times to the sub-Lorentzian case.

The work is organized in the following way. In Section~\ref{sec:basic} we review the main definitions of Lorentzian geometry and introduce them for sub-Lorentzian manifolds. In Section~\ref{sec:AlexTopology} we study the reachable sets and the causal structure of sub-Lorentzian manifolds, which allows us to introduce the Alexandrov topology. Section~\ref{sec:AlexTopology} contains our main result, where we compare the Alexandrov topology with the manifold topology. We also present different ways of introducing the Alexandrov topology and study some special class of sub-Lorentzian manifolds, that we call {\it chronologically open}, in which the Alexandrov topology behaves similarly as in the classical Lorentzian case. The last Section~\ref{sec:Separation} is devoted to the study of the Alexandrov topology and the time separation topology.


\section{Basic Concepts}\label{sec:basic}

\subsection{Lorentzian Geometry}
\begin{definition}
Let $M$ be a smooth manifold and $g$ be a smoothly varying $(0,2)$-tensor with one negative eigenvalue on $M$.\footnote{Note that we cannot define the eigenvalues of a quadratic form, but their sign due to Sylvester's Theorem of Inertia. We call the number of negative eigenvalues the index of the form.} Then the pair $(M,g)$ is called a Lorentzian manifold. We assume also that $M$ is connected throughout the paper. 
\end{definition}
 
The symmetric torsion free Levi-Civita connection, normal neighbourhoods and exponential maps are defined as in the Riemannian geometry. We refer the reader to~\cite{ON} for notations and the main definitions and results.
\begin{definition}
Let $(M,g)$ be a Lorentzian manifold and $p\in M$. A vector $v\in T_pM$ is called
\begin{itemize}
\item spacelike, if $g(v,v)>0$ or $v=0$,
\item null or lightlike, if $g(v,v) = 0$ and $v\neq 0$,
\item timelike, if $g(v,v) < 0$,
\item nonspacelike, if $v$ is null or timelike.
\end{itemize}
A vector field $V$ is called timelike, if $V_p$ is timelike for all $p\in M$, and similarly for the other conditions.\end{definition}

\begin{definition}
Let $(M,g)$ be a Lorentzian manifold. A globally defined timelike vector field $T$ is called a time orientation of $(M,g)$. The triplet $(M,g,T)$ is known as space-time or time oriented manifold.
\end{definition}
Every Lorentzian manifold is either time orientable or admits a twofold time orientable cover~\cite{BEE96}. Curves on a space-time are distinguished according to their causal nature and time orientation as stated in the following definition.

\begin{definition}
Let $(M,g,T)$ be a space-time. An absolutely continuous curve $\gamma\colon I\to M$~is 
\begin{itemize}
\item future directed, if $g(T,\dot\gamma)<0$ almost everywhere,
\item past directed, if $g(T,\dot\gamma)>0$ almost everywhere.
\end{itemize}
We call an absolutely continuous curve $\gamma\colon I\to M$ null (timelike or nonspacelike) and future or past directed, if $g(\dot\gamma,\dot\gamma)= 0$ ($<0$ or $\leq 0$) almost everywhere and $\gamma$ is future directed or past directed. We call $\gamma$ simply null (timelike or nonspacelike), if it is null (timelike or nonspacelike) and either past or future directed.
\end{definition}
We abbreviate timelike to t., nonspacelike to nspc., future directed to f.d., and past directed to p.d.\ from now on. Nspc.f.d.\ curves are also called {\it causal}.

\begin{definition}
Let $M$ be a space-time and $p,q\in M$. We write
\begin{itemize}
\item $p\leq q$ if $p=q$ or there exists an absolutely continuous nspc.f.d.\ curve from $p$ to $q$,
\item $p\ll q$ if there exists an absolutely continuous t.f.d.\ curve from $p$ to $q$.
\end{itemize}
Define the chronological past $I^-$, future $I^+$, the causal past $J^-$ and future $J^+$ of $p\in M$~by
\begin{align*}
I^+(p) &= \{q\in M \:|\:p \ll q\},\qquad
I^-(p) = \{q\in M \:|\:q \ll p\},\\
J^+(p) &= \{q\in M \:|\:p \leq q\},\qquad
J^-(p) = \{q\in M \:|\:q \leq p\}.
\end{align*}
\end{definition}
Let $U\subset M$ be an open set. Then we write $\ll_U$ for the causal relation $\ll$ taken in $U$, where $U$ is considered as a manifold itself, and  $I^+(p,U)$ for the future set obtained on it. Observe, that usually $I^+(p,U) \neq I^+(p)\cap U$ since $U$ might lack convexity.

In~\cite[Chapter 5, Proposition 34]{ON} it is shown that if $U$ is a normal neighbourhood of $p$, then $I^+(p,U) = \exp_p(I^+(0)\cap V_p)$ where $I^+(0)\subset T_pM$ is the Minkowski light cone and $V_p$ is the neighbourhood of the origin of $T_pM$ on which $\exp_p$ is a diffeomorphism to $U$.

The definition of the order immediately gives that $\leq,\ll$ are transitive and $p\ll q \Ra p\leq q$. We state the following result on stronger transitivity of these relations.

\begin{proposition}~\cite{ON,Penr}\label{transitiv} The following is true for space-times.
\begin{itemize} 
\item[1)]{If either $p\leq r, r \ll q$ or $p\ll r, r\leq q$, then $p\ll q$.}
\item [2)] {Let $\gamma\colon [0,1]\to M$ be a nonspacelike curve, $p=\gamma(0)$, $q=\gamma(1)$. If $\gamma$ is not a null geodesic (up to reparametrization), then there is a timelike $\sigma\colon [0,1]\to M$ such that $\sigma(0) = p$ and $\sigma(1) = q$.}
\end{itemize}
\end{proposition}
Lorentzian manifolds can be categorized in different levels of a causal hierarchy.

\begin{definition}\label{def causalities}
 We call a space-time $(M, g, T)$
\begin{itemize}
\item[1)] chronological, if there is no $p\in M$ such that $p\ll p$,
\item[2)] causal, if there are no two points $p, q \in M$, $p\neq q$, such that $p\leq q\leq p$,
\item[3)] strongly causal, if for every point $p$ and every open neighbourhood $U$ of $p$ there is a neighbourhood $V\subset U$ of $p$, such that no nspc.\ curve that leaves the neighbourhood $V$ ever returns to it,
\end{itemize}
\end{definition}
Neighbourhoods $V$ in the definition of strongly causal manifolds are called causally convex. Each requirement in Definition~\ref{def causalities} is stronger than the preceding one and no two requirements are equivalent, see~\cite{BEE96}. An example of strongly causal manifold is any convex neighbourhood $U$ of a point with compact closure. 

The time-separation function, or the Lorentzian distance function, is defined by
\begin{equation}\label{eq:time_separation}
\ts(p,q) = \sup\left\{\int_0^1\sqrt{-g(\dot\gamma,\dot\gamma)\:}dt\:|\:\gamma \in \Omega_{p,q}\right\},
\end{equation}
where the space $\Omega_{p,q}$ consists of future directed nonspacelike curves defined on the unit interval joining $p$ with $q$. If $\Omega_{p,q}=\emptyset$ the we declare the supremum is equal to $0$. Because the arc length
$L(\gamma) = \int_0^1\sqrt{-g(\dot\gamma,\dot\gamma)}\,dt$
of the curve $\gamma$ is invariant under monotone reparameterization, the normalization to the unit interval is admissible. It follows immediately from the definition that $\ts$ satisfies the inverse triangle inequality
\[
\ts(p,q) \geq \ts(p,r) + \ts(r,q)\quad\text{for all points} \quad p\leq r\leq q\in M. 
\]
However, the distance function fails to be symmetric, possibly even to be finite, and it vanishes outside the causal future set. 


\subsection{Sub-Lorentzian Manifolds}


The setting we are going to explore now is a generalization of the Lorentzian geometry. On  a sub-Lorentzian manifold a metric, that is a non-degenerate scalar product at each point, varying smoothly on the manifold, is defined only on a subspace of the tangent space, but not necessarily on the whole tangent space. If the subspace is proper, those manifolds may behave quite differently from Lorentzian ones.

\begin{definition}
A smooth distribution $D$ on a manifold $M$ is a collection of subspaces $D_p\subset T_pM$ at every point $p\in M$ such that for any point $p\in M$ there exists a neighbourhood $U$ and smooth vector fields $X_1,...,X_k$ satisfying $D_q = \mathrm{span}\{ X_{1}(q),...,X_{k}(q)\}$ for all $q\in U$. We write $\rank_p(D) = \dim(D_p)$.
\end{definition}
We assume that $2\leq k=\rank_p(D) <\dim(M)=n$ everywhere. 
An {\it admissible or horizontal} curve $\gamma\colon I\to M$ is an absolutely continuous curve such that $\dot\gamma(t)\in D_{\gamma(t)}$ almost everywhere and such that $\dot\gamma$ is locally square integrable with respect to an auxiliary Riemannian metric. We are interested in whether two arbitrary points can be connected by an admissible curve. That need not be possible due to the Frobenius theorem.

\begin{definition}
Let $N\in\N$ and $I=(i_1,...,i_N)\in\N^N$ be a multi-index. Let $X_1,...,X_k$ be vector fields on $M$. We define
$X_I = [X_{i_1},[X_{i_2},...[X_{i_{N-1}},X_{i_N}]]]$.
A distribution $D$ satisfies the bracket-generating hypothesis, if for every point $p\in M$ there are $N(p)\in\N$, a neighbourhood $U$ of $p$ such that  $D_q = \mathrm{span}\{ X_{1}(q),...,X_{k}(q)\}$ for all $q\in U$ and 
\[
T_pM = \mathrm{span}\{X_I(p) \:|\: I=(i_1,...,i_N),\ \ N\leq N(p)\}.
\]
\end{definition}

A sufficient condition of the connectivity by admissible curves is given by the Chow-Rashevskii theorem~\cite{C,M,R}, stating that if $M$ is a connected manifold with a bracket-generating distribution $D$, then any two points $p,q\in M$ can be connected by an admissible curve.

\begin{definition}
A sub-Lorentzian manifold is a triple $(M,D,g)$ where $M$ is a manifold with a smooth bracket generating distribution $D$ and a non-degenerate symmetric bilinear form $g\colon D_p\times D_p \to \R$ of constant index $1$ smoothly varying on $M$. 

A sub-space-time is a quadruple $(M,D,g,T)$, where $(M,D,g)$ is a sub-Lorentzian manifold and $T$ is a globally defined horizontal timelike vector field. Like in the Lorentzian case we call $T$ a time orientation.
\end{definition}

Similar to Lorentzian geometry, in a sub-Lorentzian manifold we have the order relations $\leq$, $\ll$, satisfying the following properties.

\noindent $\bullet$ $\leq,\ll$ are partial orders. Note that the first property in Proposition~\ref{transitiv} does not hold anymore for sub-space-times, see Example~\ref{weirdstuff1}. However, it still holds for a smaller class of sub-space-times, that we call chronologically open sub-space-times, see Definition~\ref{def:precausal}.

\noindent $\bullet$ Causality conditions are defined analogously to the Lorentzian case. 

\noindent $\bullet$  The sets $J^\pm, I^\pm$ for sub-Lorentzian manifolds are defined as before for Lorentzian ones. 

The following statement can be proved by the same arguments as in the sub-Riemannian case, see, for instance~\cite{M}.
\begin{proposition}\label{extension}
A sub-Lorentzian metric $g$ on a sub-space-time $(M,D,g,T)$ can always be extended to a Lorentzian metric $\tilde g$ over the whole manifold $M$. A sub-space-time $(M,D,g,T)$ thus becomes a space-time $(M,\tilde g,T)$ with the same time orientation $T$.
\end{proposition}

Proposition~\ref{extension} allows us to use results from Lorentzian geometry, because every horizontal t.f.d.\ curve will be t.f.d.\ with respect to all extended metrics. Moreover, using extended metrics $\tilde g_\lambda = g\circ\pi_D + \lambda^2\:h\circ\pi_{D^\bot}$ for some Riemannian metric $h$, one can show by letting $\lambda\to \infty$, that a curve, which is t.f.d.\ with respect to all extended metrics $\tilde g_\lambda$, is actually t.f.d.\ horizontal.

Let $(M,D,g)$ be a sub-Lorentzian manifold. An open subset $U\subset M$ is called convex, if $U$ has compact closure $\bar U$ and there is an extension $\tilde g$ of $g$ and an open set $V\supset\bar U$ such that both $V$ and $U$ are uniformly normal neighbourhoods of  their points in $(M,\tilde g)$ in the sense of Lorentzian geometry. Thus one can introduce coordinates and a Lorentzian orthonormal frame $\{T,X_1,...,X_d\}$ on $V$. Having a Lorentzian metric one can construct an auxiliary Riemannian metric as in~\cite{BEE96}. Then, using Proposition~\ref{extension}, we adapt the following statement for continuous causal curves, i.e.\ continuous curves $\gamma$ satisfying $s< t\Ra \gamma(s)\leq \gamma(t)$, $\gamma(s)\neq\gamma(t)$:

\begin{proposition}\cite{BEE96}\label{Lipschitz}
Let $M$ be a sub-space-time and $\gamma\colon I\to M$ a continuous causal curve. Then $\gamma$ is locally Lipschitz with respect to an auxiliary Riemannian metric.
\end{proposition}

Using the Rademacher theorem we obtain: 

\begin{corollary}\label{regularity}
Let $M$ be a sub-space-time. A continuous causal curve is absolutely continuous and its velocity vector almost everywhere is nspc.f.d.\ with respect to an extension of the sub-Lorentzian metric and square integrable with respect to an auxiliary Riemannian metric. It follows that the curve $\gamma$ is horizontal.
\end{corollary}

Lorentzian and sub-Lorentzian manifolds do not carry a natural metric distance function which would allow a natural topology on curves in the manifold. Since monotone reparametrization does not influence causal character, we define the $C^0$-topology in the following way.

\begin{definition}
Let $U,V,W$ be open sets in topology $\tau$ of $M$ such that $V,W\subset U$. Then we define the set
\[B_{U,V,W,0,1} = \left\{\gamma\in C([0,1],M)\vert\ \gamma(0)\in V,\  \gamma(1)\in W,\ \gamma([0,1])\subset U\right\}\]
and to eliminate the need to fix parametrization we take the union over all possible parameterizaitions $B_{U,V,W} = \bigcup B_{U,V,W,0,1}$.
The $C^0$-topology on curves is the topology generated by the basis
\[
\mathscr{B} := \left\{B_{U,V,W} \:|\:U,V,W\in\tau, V,W\subset U\right\}.
\]
\end{definition}

The $C^0$-topology is constructed in such a way that curves $\gamma_n\colon [0,1]\to M$ converge to $\gamma\colon[0,1]\to\R$ if and only if
\[
\gamma_n(0)\to \gamma(0),\quad\gamma_n(1)\to\gamma(1),
\]
and for any $U\in \tau$ and $\gamma([0,1])\subset U$ there exists a positive integer $N$ such that $\gamma_n([0,1])\subset U$ for all $n\geq N$.
For general space-times or sub-space-times this notion of convergence might not be too powerful, some information on that may be found in~\cite{BEE96}. However, it becomes useful for strongly causal sub-space-times.

\begin{theorem}\label{lem c0}
If nspc.f.d. horizontal curves $\gamma_n\colon [0,1]\to M$ converge to $\gamma\colon [0,1]\to M$ in the $C^0$-topology on curves in a strongly causal sub-space-time then $\gamma$ is horizontal nspc.f.d.
\end{theorem}

\begin{proof}
The result is standard in Lorentzian geometry. A proof, that $\gamma$ is locally nspc.f.d. can be found in~\cite{HawEllis}. By the transitivity of $\leq$, it is also globaly nspc.f.d.. Using Proposition~\ref{extension} and Corollary~\ref{regularity}, we easily generalize the result to sub-Lorentzian manifolds.
\end{proof}


\section{Reachable Sets, Causality and the Alexandrov Topology}\label{sec:AlexTopology}


\subsection{Reachable Sets}


As was mentioned above, the manifold topology and the metric topology of a Riemannain or sub-Riemannian manifold are equivalent as follows from the Ball-Box theorem. The causal structure of Lorentzian manifolds allows us to introduce a new topology, called the Alexandrov topology that is (often strictly) coarser than the manifold topology. We are interested in comparing an analogue of the Alexandrov topology in sub-Lorentzian manifolds with the initial manifold topology. Let us begin by  reviewing the background from Lorentzian manifolds. 

It is well known that in a space-time (not a sub-space-time) $(M,g,T)$ the sets $I^+(p)$ and $I^-(p)$ are open in the manifold topology for all points $ p\in M$, see \cite{BEE96}. In sub-space-times this is not true anymore. We give two examples showing that sets $I^+(0)$, $I^-(0)$ may or may not be open in sub-space-times.
\begin{example}\label{specific example1}~\cite{G}
Let 
$M=\R^3=\{(x,y,z)\}$, $D=\mathrm{span}\left\{T=\partial_y + x^2\partial_z, X= \partial_x\right\}
$ 
and $g$ be the sub-Lorentzian metric determined by $g(T,T) = -1$, $g(T,X)=0$, $g(X,X)=1$. As $\partial_z = \half\left[X,[X,T]\right]$, the distribution $D$ is bracket generating. We choose $T$ as the time orientation and show that the set $I^+(0,U)$ is not open for any neighbourhood $U$ of the origin. Precisely, we show that for small enough $\theta>0$ the point $(0,\theta,0)$ will be contained in $I^+(0,U)$, while the point $(0,\theta,-a)$ will not be in $I^+(0,U)$ for any  $a>0$.

The curve $\gamma(t) = (0,t,0)$ is horizontal t.f.d.\ since $\dot\gamma(t) = T_{\gamma(t)}$, and for small enough times it runs in $U$, so $(0,\theta,0)\in I^+(p,U)$.

Assume that there is a horizontal nspc.f.d.\ curve $\sigma\colon [0,\tau]\to M$, $\sigma= (x,y,z)$, from $0$ to $(0,\theta,-a)$ for some $a>0$. Then
\begin{equation*}
\dot\sigma(t) = \alpha(t)T_{\sigma(t)} + \beta(t)X_{\sigma(t)}= (\beta(t),\alpha(t),x^2(t)\alpha(t)).
\end{equation*}
Since $\sigma$ is future directed, we find that $\alpha(t)>0$ almost everywhere. Then due to absolute continuity
\[
-a = z(\tau) = \int_0^\tau\dot z(t) dt = \int_0^\tau\alpha(t)x^2(t) dt,
\]
which is impossible since the integrand is non-negative almost everywhere. Hence $(0,\theta,-a)\notin I^+(0,U)$ for any $a>0$, which implies that $I^+(0,U)$ is not open in $M$.

The curve $\gamma$ lies on the boundary of $I^+(0,U)$. It is an example of a rigid curve, or a curve that cannot be obtained by any variation with fixed endpoints, see~\cite{LS,M,M94}. The curve $\gamma$ is the unique (up to reparametrization) horizontal nspc.f.d.\ curve from $0$ to~$(0,\theta,0)$. 
\end{example}

The next example shows that there are sub-space-times for which $I^\pm(p)$ are open. 
\begin{example}\label{Heisenberg}~\cite{G}
Consider the sub-space-time with $M=\R^3$, $D=\mathrm{span}\{X,Y\}$, where
\[X= \frac{\partial}{\partial x} + \half\:y\:\frac{\partial}{\partial z}, \qquad Y=\frac{\partial}{\partial y} - \half\:x\:\frac{\partial}{\partial z},\]
the metric $g(X,X) = -g(Y,Y) = -1$, $g(X,Y)=0$, and the vector field $X$ as time orientation. This sub-space-time is called the Lorentzian Heisenberg group. In this sub-space-time the sets $I^+(p)$, $ I^-(p)$ are open for all $p\in M$.
The details of the proof can be found in~\cite{G}, where the chronological future set of the origin
\[I^+(0) = \left\{(x,y,z)\:|\:-x^2+y^2+4|z| < 0,\ x>0\right\}\]
is calculated. The set $I^+(0)$ is obviously open. We apply the Heisenberg group multiplication in order to translate the chronological future set $I^+(p_0)$ of an arbitrary point $p_0=(x_0,y_0,z_0)\in \mathbb R^3$ to the set $I^+(0)$. The map
\[\Phi_{}(x,y,z) = \left(x-x_0, y-y_0, z-z_0 + \half(y\:x_0 - x\: y_0)\right)\]
maps $p_0$ to $0$, preserves the vector fields $X$ and $Y$ and hence maps $I^+(p_0)$ to $I^+(0)$. Since $\Phi$ is also a diffeomorphism of $\R^3$, we conclude that $I^+(p_0)$ is open for all $p_0\in\R^3$. Similarly $\Phi(x,y,z) = (-x,-y,z)$ exchanges $X$ for $-X$ and $Y$ for $-Y$, i.e.\ it preserves the distribution and the scalar product, but it reverses time orientation. Hence, it maps $I^-(p_1,p_2,p_3)$ to $I^+(-p_1,-p_2,p_3)$ which proves, that also all chronological past sets are open.
\end{example}

In the following definition we generalize the properties of the above mentioned map~$\Phi$.

\begin{definition}\label{sublorentzian isometry}
Let $(M,D,g,T)$ be a sub-space-time. A diffeomorphism $\Phi\colon M\to M$ is called a sub-Lorentzian isometry, if it preserves the sub-Lorentzian causal structure, i.e.
$$
\Phi_*D = D,\qquad
\Phi^*g = g,\qquad
g(\Phi_*T,T) < 0.
$$
\end{definition}

\begin{lemma}\label{quotient}
Let $G$ be a group with a properly discontinuous action by sub-Lorentzian isometries on a sub-space-time $M$. Then the quotient $M/G$ carries a canonical sub-space-time structure, such that the projection to equivalent classes $\pi\colon M\to M/G$ is a covering map and a local sub-Lorentzian isometry. Furthermore
$
I^+_{M/G}\big(\pi(p)\big) = \pi\big(I^+_M(p)\big)$.
\end{lemma}

\begin{proof}
The proof that $M/G$ is a smooth manifold and $\pi$ is a covering map can be found in any textbook on differential topology. Take $q\in M/G$ and $p\in\pi\inv(q)$. Then we define $D_q = \{v\in T_q(M/G)\:|\:\exists\ V\in D_p:\pi_*V=v\}$. As the action of $G$ preserves the distribution, $D_q$ is defined independently of the choice of $p\in\pi\inv(q)$. As the commutators of $\pi$-related vector-fields are $\pi$-related and $\pi$ is a local diffeomorphism, the distribution on $M/G$ is bracket-generating.

Now take $q\in M/G$ and $v,w\in T_q(M/G)$. We define $g_q(v,w):=g_p(V,W)$, where $p\in\pi\inv(q)$, $V,W\in T_pM$ and $\pi_*V=v$, $\pi_*W=w$. As $\pi$ is an isomorphism, $V$ and $W$ are uniquely defined, and as the action of $G$ preserves the scalar product, the value we get is independent of the choice of $p$. This gives a well-defined sub-Lorentzian metric on $M/G$.

The construction of a time orientation on $M/G$ is a little subtler. If all isometries $\Phi\in G$ satisfy $\Phi_*T = T$ (as is the case in our examples), we can set $T^{M/G} = \pi_*T$. If $G$ is finite, we choose $T^{M/G}_q = \sum_{p\in \pi\inv(q)}\pi_*T_p$. 

In the general case in order to define time orientation $T^{M/G}$ it is preferable to work with an alternative definition of a time-orientation. A time orientation $\tilde T$ is a set-valued map $\tilde T\colon M\to \mathscr{P}(TM)$ to the power set $\mathscr{P}(TM)$ of $TM$ such that
\begin{enumerate}
\item $\tilde T(p)\subset T_pM$ is connected,
\item the set of timelike vectors in $T_pM$ equals $\tilde T(p)\cup -\tilde T(p)$, and
\item $\tilde T$ is continuous in the sense that if $U\subset M$ is open and connected and a vector  field $X$ is timelike on $U$, then either $X_p\in \tilde T(p)$ or $X_p\in -\tilde T(p)$ for all $p\in U$.
\end{enumerate}

So a time orientation is a continuous choice of a time cone in the tangent space as the future time cone. Clearly, a time orientation $T$ by a vector field defines a time orientation in the later sense by choosing $\tilde T(p) = \{v\in T_pM\:|\:g(v,v)<0,\ g(v,T)<0\}$, but also a time orientation map defines a time orienting vector field, as we can define local time orientation vector fields on charts and piece them together with a smooth partition of unity. Owing to the fact that we always take our local time orientation vector fields from the same time cone, their convex combination by the partition of unity will give a global timelike vector field, since time cones are convex in the sense of subsets of vector spaces.

A time orientation in the set-valued sense on $M/G$ is defined by $$\tilde T(q) = \{v\in T_q(M/G) \:|\:g_q(v,v)<0,\ g_p(V,T_p)<0\}$$ for any $p\in \pi\inv(q)$ and $V\in T_pM$ such that $\pi_*V = v$. This is well defined by the property $g(\Phi_*T,T)<0$ for all $\Phi\in G$.

So there is a canonical well-defined sub-Lorentzian structure on $M/G$ and the quotient map $\pi$ is a local isometry by construction. The properties of a curve of being timelike or future directed are entirely local, so under the map $\pi$ t.f.d.\ curves lift and project to t.f.d.\ curves. This implies the identity of the chronological future sets.
\end{proof}

The Examples~\ref{specific example1} and~\ref{Heisenberg} lead to the consideration of a special type of sub-space-times. 

\begin{definition}\label{def:precausal}
A sub-space-time in which $I^\pm(p)$ are open for all $p\in M$ is called chronologically open. 
\end{definition}
Except of some special cases, such as the Minkowski space, in general $J^+(p)$ and $J^-(p)$ are not closed for Lorentzian manifolds. For example, if the point $(1,1)$ in two dimensional Minkowski space is removed, then $J^+(0,0)$ and $J^-(2,2)$ are not closed. The most what is known up to now is the following.

\begin{proposition}~\cite{G1}\label{properties of sets}
Let $(M,D,g,T)$ be a sub-space-time, $p\in M$ and $U$ a convex neighbourhood of $p$. Then 
\begin{enumerate} 
\item $\overline{\mathrm{int}(I^+(p,U))}^U = J^+(p,U)$. In particular, $\mathrm{int}\:(I^+(p,U)) \neq\emptyset$ and $J^\pm(p,U)$ is closed in $U$. It holds globally that $J^+(p)\subset\overline{\mathrm{int}(I^+(p))}$.
\item $\mathrm{int}(I^+(p,U)) = \mathrm{int}(J^+(p,U))$ and $\mathrm{int}(I^+(p)) = \mathrm{int}(J^+(p))$.
\item $\tilde\partial I^+(p,U) = \tilde\partial J^+(p,U)$ and $\partial I^+(p) = \partial J^+(p)$.
\end{enumerate}
Here $\overline{\:A\:}^U$ is the closure of $A$ relative to $U$ and $\tilde\partial A$ is the boundary of $A$ relative to $U$. 
\end{proposition}


\subsection{The Alexandrov Topology}


\begin{lemma}\label{precausal sub-space-time}
In a chronologically open sub-space-time $(M,D,g,T)$ the set
\[\mathscr{B} = \{I^+(p)\cap I^-(q)\:|\:p,q\in M\}\]
is the basis of a topology.
\end{lemma}

\begin{proof}
We have to check two requirements:
\begin{enumerate}
\item For each $p\in M$ there is a set $B\in \mathscr{B}$ such that $p\in B$.
\item If $p\in B_1\cap B_2$ then there is $B_3$ such that $p\in B_3 \subset B_1\cap B_2$, where $B_1,B_2,B_3\in\mathscr {B}$.
\end{enumerate}

To show the first statement we take a t.f.d.\ horizontal curve $\gamma\colon (-1,1)\to M$ such that $\gamma(0)=p$. Then $p\in I^+(\gamma(-1))\cap I^-(\gamma(1))$.

For the second statement we denote by $B_1=(I^+(q)\cap I^-(r))$, $B_2=(I^+(q')\cap I^-(r'))$ and chose a point $p\in B_1\cap B_2$. The set $B_1\cap B_2$ is an open neighbourhood of $p$ in the manifold topology. Thus there is a t.f.d.\ curve $\gamma\colon (-2\epsilon,2\epsilon)\to B_1\cap B_2$ with $\gamma(0)=p$. 
Then we have
\[r,r'\ll \gamma(-\epsilon)\ll p\ll \gamma(\epsilon) \ll q,q',
\]
or, in other words $p\in I^+(\gamma(-\epsilon))\cap I^-(\gamma(\epsilon)):=B_3\subset B_1\cap B_2$.
\end{proof}

\begin{definition}
Let $(M,D,g,T)$ be a sub-space-time. The Alexandrov topology $\mathcal A$ on $M$ is the topology generated by the subbasis $\mathscr{S}=\left\{I^+(p), I^-(p)\:|\:p\in M\right\}$.
\end{definition}
 Lemma~\ref{precausal sub-space-time} implies that for chronologically open space-times the set
\[\mathscr{B} = \{I^+(p)\cap I^-(q)\:|\:p,q\in M\}
\]
is a basis of the Alexandrov topology $\mathcal A$.
\begin{remark}
Note that, while we need a time orientation $T$ to define the Alexandrov topology, $\mathscr{A}$ is independent of $T$. A choice of a different time-orientation can at most reverse future and past orientation of the manifold. Note also, that since $I^+(p)$, $I^-(p)$ are open for space-times, the manifold topology is always finer than the Alexandrov topology. This is only true for chronologically open sub-space-times in the sub-Lorentzian case.
\end{remark}


\subsection{Links to Causality}


The definition of the Alexandrov topology suggests a link between the Alexandrov topology and causal structure of a sub-space-time. The following theorem generalizes a well known fact from the Lorentzian geometry.

\begin{theorem}\label{compact spacetime}
Every sub-space-time compact in the Alexandrov topology (in particular every compact space-time) $M$ fails to be  chronological.
\end{theorem}

\begin{proof}
Cover the manifold $M$ by $\mathscr{U} = \{I^+(p)\:|\:p\in M\}$. Then we can extract a final subcover $I^+(p_1),...,I^+(p_n)$ of $M$. Since these sets cover $M$, for any index $k$ with $1\leq k\leq n$ there exists $i(k)$, $1\leq i(k)\leq n$ such that $p_k\in I^+(p_{i(k)})$. It gives an infinite sequence $p_1, p_{i(1)}, p_{i(i(1))},...$ containing only finitely many elements. Thus there exists $k$, $1\leq k\leq n$, such that $p_k$ appears more than once in the sequence. This means that $p_k \ll p_k$ by transitivity of the relation $\ll$.
\end{proof}

\begin{proposition}\label{spacetime Hausdorff}{\cite{BEE96, Penr}}
For any space-time $(M,g,T)$ the following are equivalent
\begin{enumerate}
\item \label{itm Hausdorff} The Alexandrov topology on $M$ is Hausdorff.
\item \label{itm strongly causal} The space-time $M$ is strongly causal.
\item \label{itm manifold}The Alexandrov topology is finer than the manifold topology.
\end{enumerate}
\end{proposition}

The proof of Proposition \ref{spacetime Hausdorff} employs quite a few Lorentzian notions, specifically openness of the sets $I^\pm$ and strong transitivity of the causal relations $\leq, \ll$, that do not generally hold in the sub-Lorentzian case. We show that results of Proposition~\ref{spacetime Hausdorff} still hold in chronologically open sub-space-times, see also Theorem~\ref{th: cc-transitive}.

\begin{lemma}\label{lem boundary}
Let $(M, D,g,T)$ be a sub-space-time, $p\in M$ and $q\in\partial J^+(p)\cap J^+(p)$, $p\neq q$. Then any nspc.f.d.\ curve $\gamma$ from $p$ to $q$ is totally contained in $\partial J^+(p)$. If $\gamma$ is a nspc.f.d.\ curve with $\gamma(\theta) \in \mathrm{int}(J^+(p))$ then $\gamma(t)\in \mathrm{int}(J^+(p))$ for all $t\geq \theta$.
\end{lemma}

\begin{proof}
Let $\gamma\colon I\to M$ be an absolutely continuous nspc.f.d.\ curve and $p=\gamma(0)$, $q=\gamma(\theta)$. Assume that $q\in \partial J^+(p)$. Let us choose a normal neighbourhood $U$ of $q$ and an orthonormal frame $T,X_1,...,X_d$ on $U$. Then we can expand $\dot\gamma = T + \sum_{i=0}^nu^iX_i$
in that frame, where we fixed the parametrisation.
The theory of ordinary differential equations ensures that there is a neighbourhood $V\subset U$ of $q$ and $\epsilon>0$ such that a unique solution of the Cauchy system
\[
\gamma_r(0) = r, \qquad \dot\gamma_r(t) = -\Big(T(\gamma_r(t))+\sum_{i=0}^nu^i(t)\: X_{i}(\gamma_r(t))\Big),\quad t\in [0,\epsilon],
\]
exists for continuous coefficients $u^i$, for any $r\in V$, and that $\gamma_r(\epsilon)$ depends continuously on $r$. The result can also be extended for $u^i\in L^2([0,\epsilon])$.

 As $q\in\partial J^+(p)$, there is $r\in V\cap (J^+(p))^c$. Since $\gamma_r$ is nspc.p.d., $\gamma_r(t)\notin J^+(p)$ for any $t\in [0,\epsilon]$ either, but as $\gamma_r(t)$ depends continuously on the initial data $r\in J^+(p)^c$, we can choose $r$ such that $\gamma_r(t)$ comes arbitrarily close to $\gamma(\theta-t)$. Thus, if $\gamma(\theta)\in \partial J^+(p)\cap J^+(p)$ for $\theta>0$, then also $\gamma(t)\in \partial J^+(p)\cap J^+(p)$ for $t\in [\theta-\epsilon,\theta]$ for some $\epsilon>0$. This implies $\gamma[0,\theta]\subset  J^+(p)\cap \partial J^+(p)$.
\end{proof}

\begin{theorem}\label{th:precausal_transitivity}
Let $(M,D,g,T)$ be a chronologically open sub-space-time. Then $p\leq q\ll r$ or $p\ll q\leq r$ implies $p\ll r$.
\end{theorem}

\begin{proof}
Assume $p\ll q\leq r$. Then $q\in I^+(p) = \mathrm{int}(J^+(p))$. Therefore, any causal curve $\gamma$ connecting $p$ to $r$ and passing through $q$ is contained in $\mathrm{int}(J^+(p))$ after crossing $q$ due to Lemma~\ref{lem boundary}. In particular $r\in \mathrm{int}(J^+(p)) = I^+(p)$.

Assume now $p\leq q \ll r$. Reversing the time orientation does not change the Alexandrov topology, so the sub-space-time remains chronologically open. In $(M,D,g,-T)$ we have $r\ll q \leq p$, so by the first step $r\ll p$, and by reversing $T$ again we have $p\ll r$ in $(M,D,g,T)$.
\end{proof}

Note that the proof of Theorem~\ref{th:precausal_transitivity} uses Proposition~\ref{properties of sets} instead of the calculus of variations as in the classical Lorentzian case. It shows that the strong transitivity of causal relations holds not because one can vary curves, but because the system of t.f.d.\ curves has nice properties. Combining the last two results, we state the following summary.

\begin{corollary}\label{Alexandrov Hausdorff}
Let $(M,D,g,T)$ be a sub-space-time. If $M$ is strongly causal, its Alexandrov topology is finer than the manifold topology. If $M$ is chronologically open, the equivalences formulated in Proposition~\ref{spacetime Hausdorff} still hold.
\end{corollary}

\begin{proof}
Let  $M$ be strongly causal, but not necessarily chronologically open. Let $U\in\tau$ be causally convex. Let $\gamma_q\colon (-2\epsilon, 2\epsilon)\to U$ be a t.f.d.\ curve such that $\gamma_q(0) = q$, $\gamma_q(-\epsilon),\gamma_q(\epsilon)\in U$. Then $q\in I^+(\gamma_q(-\epsilon))\cap I^-(\gamma_q(\epsilon))=:A_q\subset U$. Thus $U = \bigcup_{q\in U}A_q$
and any set that is open in the manifold topology, is also open in the Alexandrov topology. The rest is just a special case of Theorem~\ref{th: cc-transitive}, which we state later.
\end{proof}

\begin{example}\label{example finer than mfd}
For non-chronologically open sub-space-times it is even possible to get an Alexandrov topology that is strictly finer than the manifold topology, but that are not strongly chronological. Consider 
\[
N=\left\{\Big((x,y,z)\in\R^3 \Big)\setminus\Big((0,2n,z)\mid\ z\in\R, n\in\Z\Big)\right\}.
\]
Let the horizontal distribution $D$ be spanned by $T=\frac{\partial}{\partial y}+x^2\frac{\partial}{\partial z}$, $X=\frac{\partial}{\partial x}$ and equipped with the scalar product $g(T,T) = - g(X,X) = -1$, $g(T,X) = 0$. Now consider the action of the group $G=\{\Phi_n(x,y,z) = (x,y+2n,z)\}$ by sub-Lorentzian isometries on $N$. Define $M=N/G$ to be the quotient space of the group action and let $\pi\colon N\to M$ denote the canonical projection. Then $(M,\pi_*D,\pi_*g,\pi_*T)$ is a sub-space-time in its own right by Lemma~\ref{quotient}. To simplify the notation we write $D=\pi_*D$, $g=\pi_*g$ and $T=\pi_*T$.

Let us first show that $(M,D,g,T)$ is not strongly causal. Let $p_0 = (0,y_0,z_0)$ be some point and $U$ a neighbourhood of $p_0$. Let $0<\delta<1$ and consider the curve
\[
\gamma_\delta(t) = \pi\left(\delta t, y_0 + t, z_0 + \frac{\delta^2t^3}{3}\right).
\]
By definition, the $x$-coordinate of $\gamma_\delta$ is strictly positive for positive times, and as $p_0\in M$, the curve is well-defined on the quotient space. The curve satisfies $\gamma_\delta(0) = p_0$, $\dot\gamma_\delta(t) = T_{\gamma_\delta(t)}+\delta\:X_{\gamma_\delta(t)}$ and $\gamma_\delta(2) = \left[p_0 +  (2\delta,2, \frac{8\delta^2}{3})\right]$. So for small $U$, any curve of the t.f.d.\ family $\{\gamma_\delta\}_{\delta\in(0,1)}$ leaves $U$ before $t=2$, but for small enough $\delta$ they return to $U$ later. Hence $p_0$ does not have arbitrarily small neighbourhoods $U$ in $M$ such that no t.f.d.\ curve that once leaves $U$ will never return, and strong causality fails at $p_0$.

Now we want to show that the Alexandrov topology on $M$ is finer than the manifold topology. Take $p_0=[x_0,y_0,z_0]\in M$ and a neighbourhood $U$ of $p_0$. We want to show that there are points $p_1=[x_1,y_1,z_1]$ and $p_2=[x_2,y_2,z_2]$ such that $p_0\in I^+(p_1)\cap I^-(p_2)\subset U$.

First assume that $x_0 = 0$. Then the curve $\gamma(t) = \pi(0,y_0+t,z_0)$ is well-defined on some small parameter interval such that $y_0+t\notin 2\Z$. Furthermore, it is t.f.d.\ as $\dot\gamma = T$, and clearly for some small $\epsilon>0$: $\gamma(-\epsilon,\epsilon)\subset U$. As in Example~\ref{specific example1} we find, that up to monotone reparametrisation, $\gamma$ is the only t.f.d.\ curve from $p_1:=[0,y_0-\epsilon,z_0]$ to $p_2:=[0,y_0+\epsilon,z_0]$. Then
\[
p_0\in I^+(p_1)\cap I^-(p_2) = \{\pi(0,y_0+t,z_0)\:|\:t\in(-\epsilon,\epsilon)\} \subset U.
\]
This also shows that there are sets which are open in the Alexandrov topology, but not in the manifold topology. Let us now suppose $x_0\neq 0$ and choose $\epsilon<\frac{|x_0|}{16}$.  Consider the curve
\[
\gamma\colon(-2\epsilon,2\epsilon)\to M, \qquad \gamma(t) = \pi\left(x_0,y_0+t, z_0 + x_0^2t\right)
\]
with derivative $\dot\gamma(t) = T_{\gamma(t)}$ and set $p_1 = \gamma(-\epsilon)$, $p_2= \gamma(\epsilon)$. Due to non-vanishing value of $x_0$ the curve is well-defined. Clearly $p_0\in I^+(p_1)\cap I^-(p_2)$. Now take $\bar p = [\bar x,\bar y,\bar z]\in I^+(p_1)\cap I^-(p_2)$. The components $\bar x,\bar z$ are independent of the choice of a representative, and as $z$ is non-decreasing along t.f.d.\ curves, we find 
\[
z_1 = z_0 - x_0^2\epsilon \leq \bar z \leq z_0+x_0^2\epsilon = z_2.
\]
Let $\sigma=(\sigma_x,\sigma_y,\sigma_z)$ be a t.f.d.\ curve connecting $p_1=\sigma(0)$ with $\bar p$. Fix parametrisation of $\sigma$ such that $\dot\sigma = T + \beta X$. Then we find a fixed time $\theta>0$ such that $\bar p = \sigma(\theta)$ and
\begin{align*}
\sigma_z(t) &= z_1 + \int_0^t\dot\sigma_z(s)\:ds
	= z_1 + \int_0^t\sigma_x^2(s)\:ds
	\geq z_1 + \int_0^{\min\{t,\frac{|x_1|}{2}\}}\sigma_x^2(s)\:ds
	\\
	&\geq z_1 + \int_0^{\min\{t,\frac{|x_1|}{2}\}}\left(\frac{x_1}{2}\right)^2\:ds
	= z_1 + \min\left\{t,\frac{|x_1|}{2}\right\}\frac{x_1^2}{4}.
\end{align*}
As $x_1=x_0$ we combine the above obtained inequalities and see
\[
z_0 + x_0^2\epsilon \geq \bar z =  \sigma_z(\theta)\geq z_0 - x_0^2\epsilon + \min\left\{\theta,\frac{|x_0|}{2}\right\}\frac{x_0^2}{4}
\]
or equivalently $8\epsilon > \min\{\theta, |x_0|/2\}$,
which in its turn implies $8\epsilon>\theta$ due to our choice of~$\epsilon$. This however leads to
\[\left|\bar y - y_0\right| \leq \left|\bar y - y_1\right| + |y_1 -y_0| = \theta + \epsilon< 9\epsilon,\]
\[|\bar x-x_0|= |\bar x-x_1| = \left|\int_0^\theta \beta(s)\: ds\right|\leq \theta< 8\epsilon, \]
\[|\bar z-z_0| = \left|\int_0^\theta \sigma_x^2(s)\:ds\right| <\int_0^\theta(|x_0|+8\epsilon)^2\:ds \leq 8\epsilon\left(|x_0|+8\epsilon\right)^2 \leq 8\epsilon\:\left(\frac{3\:|x_0|}{2}\right)^2.\]
Hence, the set $I^+(p_1)\cap I^-(p_2)$ becomes arbitrarily small as we let $\epsilon$ tend to $0$, but contains $p_0$ for any $\epsilon>0$. So inside every neighbourhood $U$ of $p_0$, that is open in the manifold topology, we can find $p_1,p_2\in U$ such that $p_0\in I^+(p_1)\cap I^-(p_2)$. This means that also at $x_0\neq 0$ the Alexandrov topology is finer than the manifold topology: $\tau\subset\mathscr{A}\Ra\tau\subsetneq\mathscr{A}$.
\end{example}

\begin{theorem}\label{pullback finer}
A sub-space-time $(M,D,g,T)$ is chronological if and only if the pullback of the Alexandrov topology along all t.f.d.\ curves to their parameter inverval $I\subset \R$ is finer than the standard topology on $I$ as a subspace of $\R$.
\end{theorem}

\begin{proof}
Assume that $M$ fails to be chronological. Then there is a closed t.f.d.\ curve $\gamma\colon[0,1]\to M$, which means $\gamma(s)\ll \gamma(t)$ for all $t,s\in[0,1]$, hence the pullback of the Alexandrov topology along this specific $\gamma$ is $\{\emptyset, [0,1]\} = \gamma\inv(\mathscr{A})$.

If, on the other hand, $M$ is chronological, and $\gamma$ is an arbitrary curve on the open interval $I$, then $(s,t) = \gamma\inv\left(I^+(\gamma(s))\cap I^-(\gamma(t))\right)$ for $s,t\in I$. If $I=[a,b]$ is closed, then $(s,b] = \gamma\inv(I^+(\gamma(s))\cap I^-(p))$ for any $p\in I^+(\gamma(b))$ and similarly for the other cases. 
\end{proof}

\subsection{The Alexandrov and Manifold Topology in sub-Lorentzian Geometry}

We now come to a result that is new in sub-Lorentzian geometry and at the core of our investigation, namely the relation of the Alexandrov topology $\mathscr{A}$ to the manifold topology $\tau$ in proper sub-space-times. In contrast to space-times, we have the following result.

\begin{theorem}\label{TopologyAlexandrovTopology}
For sub-space-times all inclusions between the Alexandrov topology $\mathscr{A}$ and the manifold topology $\tau$ are possible, i.~e.\ there are sub-space-times such that
$$(1)\ \tau = \mathscr{A},\qquad (2)\ \tau \varsupsetneq \mathscr{A},\qquad (3)\ \tau \varsubsetneq \mathscr{A},\qquad (4)\ \tau \not\subset \mathscr{A},\  \tau \not\supset \mathscr{A}.
$$
\end{theorem}

\begin{proof}
 The proof is contained in the following examples.
The first case $\tau = \mathscr{A}$ in Theorem~\ref{TopologyAlexandrovTopology} is realized  for  strongly causal space-times and the second case  $\tau \varsupsetneq \mathscr{A}$ is realised in space-times, which fail to be strongly causal. They are therefore in particular satisfied for certain sub-space-times. We show even more, namely that there are sub-space-times, with a distribution $D_p\subsetneq T_pM$, $p\in M$, which exhibit this type of behaviour.

\begin{example}\label{ex:4}
In the Lorentzian Heisenberg group $M=\R^3$, $D=\mathrm{span}\{T,Y\}$, where
\[
T= \frac{\partial}{\partial x} + \half\:y\:\frac{\partial}{\partial z}, \qquad Y=\frac{\partial}{\partial y} - \half\:x\:\frac{\partial}{\partial z},
\]
and $g(T,T) = -g(Y,Y) = -1$ the manifold topology $\tau$ and the Alexandrov topology $\mathscr A $ agree: $\tau=\mathscr{A}$.

We know from Example~\ref{Heisenberg} that $\mathscr{A}\subset\tau$. To obtain the inverse inclusion $\tau\subset\mathscr{A}$ we show that the Lorentzian Heisenberg group is strongly causal. To that end we construct a family of neighbourhoods $B(p_0,\epsilon)$ of a point $p_0 = (x_0,y_0,z_0)$, that become arbitrarily small for $\epsilon\to 0$ and such that any nspc.f.d.\ horizontal curve leaving $B(p_0,\epsilon)$ will never return back. Define
\begin{equation*}
B(p_0,\epsilon) = \left\{(x,y,z)\in\R^3\Big\vert
\begin{array}{lllll}
&|x-x_0|<\epsilon,
\\
&|y-y_0| < x-x_0+\epsilon,
\\
&|z-z_0| < \frac{|x_0|+|y_0|+7\epsilon}{2}(x-x_0+\epsilon)
\end{array}
\right\}.
\end{equation*}
Obviously $B(p_0,\epsilon)$ is open in $M$ for any choice of $p_0$, $\epsilon>0$, and 
\[
B(p_0,\epsilon) \subset (x_0-\epsilon,x_0+\epsilon)\times(y_0-2\epsilon,x_0+2\epsilon)\times \left(z_0 -2C\epsilon,z_0+2C\epsilon\right),
\]
where $C := C(x_0,y_0) = \frac{|x_0|+|y_0| + 7}{2}$ for $\epsilon<1$, so $B(p_0,\epsilon)$ can be made arbitrarily small.

Let $\gamma\colon I\to M$, $0\in I$, be a nspc.f.d.\ horizontal curve such that $\gamma(0) = p_1=(x_1,y_1,z_1)\in B(p_0,\epsilon)$. Then we have $\dot\gamma(t) = \alpha(t)\:T_{\gamma(t)}+\beta(t)\:Y_{\gamma(t)}$ with $\alpha >0$ and $|\beta|\leq \alpha$. Without loss of generality we assume that $\alpha\equiv 1$ by fixing the parametrization. Then 
\begin{align*}
x(t)-x_1&=\int_0^t\alpha(s)\:ds= \int_0^t1\, ds =t,
\\
|y(t)-y_1|&=\left|\int_0^t\beta(s)\,ds\right|\leq \int_0^t1\,ds=x(t)-x_1.
\end{align*}
If we take $\tau>0$ such that $\gamma(\tau)\in B(p_0,\epsilon)$, then
\[
\tau=x(\tau)-x_1 = x(\tau)-x_0 + x_0-x_1 \leq |x(\tau)-x_0| + |x_0-x_1| < 2\epsilon,
\]
and it follows for $0\leq t\leq \tau$ that
\begin{align*}
 |z(t) & -z_1| = \left|\int_0^t\half(\alpha(s)y(s)-\beta(s)x(s))\,ds\right|
	&&\leq \half\int_0^t|y(s)|+|x(s)|\:ds\\
	&\leq \half\int_0^t|y_1|+ |y(s)-y_1| + |x_1|+|x(s)-x_1|\,ds\\
	&\leq \half\left(|y_1| + t + |x_1|+t\right)t
	&&\leq \frac{|y_1|+|x_1| + 4\epsilon}{2}(x(t)-x_0)\\
	&\leq \frac{|y_1-y_0|+|y_0| + |x_1-x_0|+|x_0|+4\epsilon}{2}(x(t)-x_0)&&\leq \frac{|x_0|+|y_0| + 7\epsilon}{2}(x(t)-x_1)
\end{align*}
since by the construction of $B(p_0,\epsilon)$ we have $|x_1-x_0|<\epsilon$, $|y_1-y_0|<2\epsilon$. Our aim now is to show that $\gamma(t)\in B(p_0,\epsilon)$, and if $\gamma$ leaves $B(p_0,\epsilon)$, then it can not return later. As $x(t)$ is increasing in $t$, we have 
$x(t) \in [x_1,x(\tau)]$
so $x(t) = \lambda x_1 + (1-\lambda)x(\tau)$ for some $\lambda\in[0,1]$ and 
\begin{equation*}
|x(t)-x_0|= \left|\lambda x_1 + (1-\lambda)x(\tau)-x_0\right|\leq\lambda |x_1-x_0| + (1-\lambda)|x(\tau)-x_0| 
<\epsilon.
\end{equation*}
Finally, we get 
\begin{align*}
|y(t) - y_0| &\leq |y(t)-y_1| + |y_1-y_0|
	< x(t) - x_1 + x_1 - x_0+\epsilon
	= x(t) - x_0+\epsilon,
\\
|z(t)-z_0| &\leq |z(t)-z_1|+|z_1-z_0|
	< \frac{|x_0|+|y_0|+7\epsilon}{2}\big(x(t)-x_0+ \epsilon\big). 
\end{align*}
So $\gamma(\tau)\in B(p_0,\epsilon)$ implies $\gamma(t)\in B(p_0,\epsilon)$ for all $t\in[0,\tau]$, and no nspc.f.d.\ curve can leave $B(p_0,\epsilon)$ and return there. 

Assume now, that there is a nspc.p.d.\ curve $\gamma$ leaving and returning to $B(p_0,\epsilon)$. By reversing the orientation of $\gamma$, we find  a nspc.f.d.\ curve leaving and returning to $B(p_0,\epsilon)$, which leads to contradiction.
Hence, we have shown that the Lorentzian Heisenberg group is strongly causal, and by that $\tau \subset \mathscr{A}$ and $\tau = \mathscr{A}$.
\end{example}

For the proof of the second statement we give the following example.

\begin{example}
Let $M$ be the Lorentzian Heisenberg group and $G = G_\varphi$ the group of isometries
\[
G = \left\{\Phi_n(x,y,z) = \left(
\begin{array}{lll}
& x-n\cosh(\varphi),
\\
& y-n\sinh(\varphi), 
\\
& z + \frac{n}{2}\big(y\cosh(\varphi) - x\sinh(\varphi)\big)
\end{array}
\right)\Big\vert\ n\in\Z\right\}
\]
for some fixed $\varphi\in\mathbb R$. Then the manifold topology $\tau_{M/G}$ in $M/G$ is strictly finer than the Alexandrov topology $\mathscr{A}_{M/G} $: $\mathscr{A}_{M/G} \varsubsetneq\tau_{M/G}$.

 We write $[p]$ for points from $M/G$.
Take a point $[q]\in I^+([p])$. We can pick up a pre-image $q\in I^+(p)\cap \pi\inv([q])$ by Lemma~\ref{quotient}. Because $I^+(p)$ is open, a small neighbourhood $U\subset M$ of $q$ is contained in $I^+(p)$. We choose $U$ to be so small, that $\pi\vert_U$ is a homeomorphism. Using  Lemma~\ref{quotient}, we find $[q]\in \pi(U)\subset \pi(I^+(p)) = I^+([p])$, where $\pi(U)$ is open, since we choose $U$ to be small enough. Therefore, $\mathscr{A}_{M/G}\subset \tau_{M/G}$.

On the other hand $M/G$ contains the closed t.f.d.\ curve
\[\gamma\colon \R\to M/G,\qquad \gamma(t) = \pi\big( t\cdot(\cosh(\varphi),\sinh(\varphi),0)\big),
\]
where $\gamma(0) = \gamma(1)$. This means, that for any $V\in\mathscr{A}$ we find $\gamma(s)\in V \LRa \gamma(t)\in V$ for any $s,t\in\R$. Therefore, the Alexandrov topology cannot be Hausdorff, in particular, $\tau_{M/G}\not\subset\mathscr{A}_{M/G}$. Together with the previous inclusion, we deduce $\mathscr{A}_{M/G} \varsubsetneq\tau_{M/G}$.
\end{example}

For the third statement we refer the reader back to Example~\ref{example finer than mfd}, see also Example~\ref{specific example1}. We  now turn to the last example to finish the proof of Theorem~\ref{TopologyAlexandrovTopology}.
\begin{example}
Take the sub-space-time from Example~\ref{specific example1} and the group
\[
G = \left\{\Phi_n(x,y,z) = (x,y+n,z)\:|\:n\in\Z\right\}.
\]
Then the Alexandrov topology and the manifold topology of $M/G$ are not comparable: $\tau \not\subset \mathscr{A}, \tau \not\supset \mathscr{A}$.

The set $I^+([0])$ is not open in $M/G$ because
from one side $[(0,\theta,0)]\in I^+([0])$ for all $\theta>0$ since $(0,\theta,0)\in I^+(0)$, but from the other side 
$[(0,\theta,-a)]\notin I^+([0])$. Indeed, if $[(0,\theta,-a)]$ were in $I^+([0])$ there would be a t.f.d.\ curve $\gamma$ in $M/G$ from $[0]$ to $[(0,\theta,-a)]$. Then we could lift  this curve to a t.f.d.\ curve $\tilde\gamma$ in $M$ from $0$ to some representative $(0,\theta+n,-a)$ in the pre-image of $[(0,\theta,-a)]$, $n\in\Z$. This is not possible as shown in Example~\ref{specific example1}.
This proves $\mathscr{A}_{M/G}\not\subset\tau_{M/G}$. 

On the other hand $M/G$ contains the closed t.f.d.\ curve
\[\gamma\colon\R\to M/G,\qquad \gamma = \pi\big(0,t,0\big),
\]
where $\gamma(0) = \gamma(1)$. This means as in the preceding example that $\tau_{M/G}\not\subset\mathscr{A}_{M/G}$.
\end{example}~\end{proof}


\subsection{The Open Causal Relations}


We have seen that pathological cases occur, when we extend the Alexandrov topology $\mathcal A$ in the obvious way to sub-space-times. We now present a different extension of $\mathcal A$ with better behaviour, that unfortunately may seem less sensible from a physical point of view, because it simply ignores the pathological cases that can occur. Let us start by explaining what happens, if we just force our extension of the Alexandrov topology to be open in the manifold topology.

\begin{lemma}\label{lem open relations}
Let $(M,D,g,T)$ be a sub-space-time. Then $q\in \mathrm{int}(I^+(p))\LRa p\in \mathrm{int}(I^-(q))$.
\end{lemma}
\begin{proof}
Assume $q\in \mathrm{int}(I^+(p))$. We have shown that also $q\in \overline{\mathrm{int}(I^-(q))}$, which means
\[
q\in \overline{\mathrm{int}(I^+(p))\cap \mathrm{int}(I^-(q))}.
\]
In particular, there is a point $r\in \mathrm{int}(I^+(p))\cap \mathrm{int}(I^-(q))$ and a t.p.d.\ curve $\gamma$ such that
\[
\gamma(0) = q, \quad\gamma\left(\half\right) = r, \quad \gamma(1) = p.
\]
By reversing time and using Lemma~\ref{lem boundary}, we see that $\gamma(t)\in \mathrm{int}(I^-(q))$ for all $t\geq \half$, so in particular $p\in \mathrm{int}(I^-(q))$. The same holds with reversed roles.
\end{proof}
We conclude, that for $p\in M$ there is some $q\in M$ such that $p\in \mathrm{int}(I^+(q))$, and we can even choose $q$ as close to $p$ as we want. In the same way as for chronologically open sub-space-times we find a way to obtain the basis of a topology, that we call {\it open Alexandrov topology}.

\begin{definition}
Let $(M,D,g,T)$ be a sub-space-time. We write $p\ll_o q$ if $q\in \mathrm{int}(I^+(p))$. The open Alexandrov topology $\mathscr{A}_o$ is the topology generated by the basis
\[
\mathscr{B}_o = \{\mathrm{int}(I^+(p))\cap \mathrm{int}(I^-(q))\:|\:p,q\in M\}.
\]
\end{definition}

\begin{theorem}
The strong transitivity of order relations $\leq,\ll_o$ hold in the sense 
\begin{enumerate}
\item $p\ll_o q\ \ \Ra\ \ p\ll q\ \ \Ra\ \ p\leq q$,
\item $p\ll_o q\leq r\ \ \Ra\ \ p\ll_o r$,
\item $p\leq q\ll_o r\ \ \Ra\ \ p\ll_o r$.
\end{enumerate}
The Alexandrov topology is finer than the open Alexandrov topology and they are the same if and only if $M$ is chronologically open.
\end{theorem}

\begin{proof}
(1) follows trivially from the definition. To show (2) we choose $p,q,r$ in $M$ such that $p\ll_o q\leq r$. Then there is a nspc.f.d.\ curve from $p$ to $r$ passing through $q$. Due to Lemma~\ref{lem boundary}, the curve is totally contained in $\mathrm{int}(J^+(p)) = \mathrm{int}(I^+(p))$ after passing $q$, in particular we find $p\ll_o r$. By reversing the time orientation and employing Lemma~\ref{lem open relations} we obtain (3).

To show that $\mathscr{A}_o\subset\mathscr{A}$ let now $p,q,r$ be such that $p\ll_o q\ll_o r$. Then there are points $p',r'$ on the connecting t.f.d.\ curves such that $p\ll_o p'\ll q \ll r'\ll_o r$, hence
\[
p\in I^+(p')\cap I^-(r') \subset \mathrm{int}(I^+(p))\cap \mathrm{int}(I^-(r)),
\]
or in other words, the Alexandrov topology is finer than the open Alexandrov topology. Clearly, two topologies agree if $M$ is chronologically open, since then $I^+(p) = \mathrm{int}(I^+(p))$ and bases agree. Nevertheless, they cannot coincide in the case when $M$ is not chronologically open.
\end{proof}

\begin{remark}
This helps us to easily generalize Theorem \ref{compact spacetime}: Every compact sub-space-time fails to be chronological, too.
\end{remark}

\begin{theorem}\label{th: cc-transitive}
Let $(M,D,g,T)$ be a sub-space-time. Then the following are equivalent:
$$
(1)\ M\ \text{ is strongly causal},\qquad(2)\ \mathscr{A}_o = \tau,\qquad (3)\ \mathscr{A}_o\ \text{ is Hausdorff}.
$$
\end{theorem}

\begin{proof}
Corollary~\ref{Alexandrov Hausdorff} implies that in strongly causal sub-space-times the Alexandrov topology is finer than the manifold topology. A combination of the same proof with Lemma~\ref{lem open relations} shows, that also the open Alexandrov topology is finer than the manifold topology, which is, in particular, Hausdorff. 

It remains to show that if in a general sub-space-time the open Alexandrov topology is Hausdorff, the sub-space-time must be strongly causal.
The proof is essentially follows~\cite{Penr}, only one needs to use the convergence of curves in the $C^0$-topology in a convex neighbourhood instead of geodesics.
\end{proof}

Since $\mathscr{A} = \mathscr{A}_o$ in chronologically open sub-space-times, we deduce Corollary~\ref{Alexandrov Hausdorff} from Theorem \ref{th: cc-transitive}. We actually find that the strongly causal sub-space-times are those, in which even the coarser topology $\mathscr{A}_o$ is Hausdorff, and not only $\mathscr{A}$.


\subsection{Chronologically Open sub-Space-Times}

We have seen in Theorem \ref{th: cc-transitive} how the open Alexandrov topology of a sub-space-time and the property of being strongly causal are linked. It is natural to ask, what other properties the topological space $(M,\mathscr{A}_o)$ possesses and whether they are also related to causality. 

Unfortunately, if $\mathscr{A}_o$ is Hausdorff, it is already the same as the manifold topology and thus metrizable, so it does not make any sense to ask for stronger properties than Hausdorff. Neither is the condition that one-point sets be closed interesting, because it only means that $M$ is chronological. The following result holds quite trivially, as $\mathscr{A}_o\subset \tau$.

\begin{proposition}
Let $(M,D,g,T)$ be a sub-space-time. Then the topological space $(M,\mathscr{A}_o)$ is second countable, path-connected and locally path-connected, i.~e.\ also first countable, separable, connected and locally connected.
\end{proposition}
Note that the same need not hold for the topological space $(M,\mathscr{A})$. In the sub-space-time from Example \ref{specific example1} we have $I^+(0,y_0,z_0) \cap I^-(0,y_0+\theta,z_0) = \{(0,y_0+t,z_0)\:|\:t\in(0,\theta)\}$, so the lines on which $z$ is constant and $x=0$ are open in the Alexandrov topology. As there are uncountably many such lines, the topology $\mathscr A$ is not second countable. Still, it is metrizable by the metric
\[
d\left(\begin{pmatrix}x\\y\\z\end{pmatrix},\begin{pmatrix}a\\b\\c\end{pmatrix}\right) = 
\begin{cases}
\min\{1,|x-a|+|y-b|+|z-c|\}&\quad\text{if}\quad x,a\neq 0,
\\ 
1&\quad\text{if}\ \ 
\begin{cases}
x=0, a\neq 0\ \text{ or }
\\
x\neq 0, a = 0,
\end{cases}
\\
1&\quad\text{if}\quad x=a=0, z\neq c,
\\ 
\min\{1,|y-b|\}&\quad\text{if}\quad x=a=0, z=c,
\end{cases}\]
and hence first countable. $(M,\mathscr A)$ is not even connected as it can be written as a disjoint union of open sets $M= \{x<0\}\cup \{x=0\} \cup \{x> 0\}$. Since chronologically open sub-space-times are the most well-behaved, we are interested in finding criteria to see that sub-space-times are chronologically open. 
\begin{definition}\cite{AgrSach,LS} 
Let $M$ be a manifold with a distribution $D$. A curve $\gamma:I\to M$ is called a Goh-curve, if there is a curve $\lambda:I\to T^*M$, such that
\[
\pi\circ\lambda = \gamma,\qquad \lambda(X) = \lambda[X,Y] = 0
\]
for all horizontal vectorfields $X,Y$.
\end{definition}

\begin{proposition}~\cite{G1} 
Let $(M,D,g,T)$ be a sub-space-time, $p\in M$, and $\gamma\colon [0,1]\to M$ a t.f.d.\ curve such that $\gamma([0,1])\subset\partial I^+(p)\cap I^+(p)$. Then $\gamma$ is a Goh-curve.
\end{proposition}
Remember, that if $q\in \partial I^+(p)\cap I^+(p)$ then any t.f.d.\ curve $\gamma$ from $p$ to $q$ must run entirely in $\partial I^+(p)$. The proof of the following proposition is trivial.

\begin{proposition}
Two step generating sub-space-times have no Goh-curves and therefore they are chronologically open.
\end{proposition}

Analogously to Theorem~\ref{pullback finer}, we can test whether a sub-space-time is chronologically open using t.f.d.\ curves.

\begin{theorem}
A sub-space-time $(M,D,g,T)$ is chronologically open if and only if for all t.f.d.\ curves $\gamma\colon I\to M$ the pullback of the Alexandrov topology along $\gamma$ to $I$ is coarser than the standard topology on $I$, considered as a subspace of $\mathbb R$.
\end{theorem}
\begin{proof}
It is clear that, if $\mathscr{A}\subset\tau_M$, then $\gamma\inv(\mathscr{A})\subset\gamma\inv(\tau_M) \subset \tau_I$, as $\gamma$ was assumed to be continuous to $(M,\tau_M)$. If on the other hand, $\mathscr{A}\not\subset \tau_M$, then there is a point $p\in M$, such that $I^+(p)$ or $I^-(p)$ is not open. Without loss of generality we may assume, that $I^+(p)$ is not open. Then there is a point $q\in I^+(p)\cap \partial I^+(p)$. 

Now take $r\in \mathrm{int}(I^-(q))$. If $r\in I^+(p)$, then $p\ll r\ll_o q \Ra q\in \mathrm{int}(I^+(p))$, which possesses a contradiction, so $r\notin I^+(p)$. In this way we can construct a sequence of points $r_1\ll r_2\ll...\ll q$ converging to $q$, such that $r_i\notin I^+(p)$. By assumption there are t.f.d.\ curves 
\[
\gamma_n\colon\left[0,2^{-n}\right]\to M, \qquad \gamma_n(0) = r_n,\qquad \gamma_n\left(2^{-n}\right) = r_{n+1}.
\]
We can consider the continuous t.f.d.\ curve $\gamma\colon[0,1)\to M$ that is obtained by 

\[\gamma(t) = \gamma_n\left(t-\sum_{k=0}^{n-1}2^{-k}\right), \qquad t\in\left[\sum_{k=0}^{n-1}2^{-k},\sum_{k=0}^n 2^{-k}\right]\]
By Corollary~\ref{regularity}, the curve $\gamma$ is absolutely continuous with nspc.f.d.\ derivative almost everywhere. Moreover, since all segments of $\gamma$ have t.f.d.\ derivatives almost everywhere, so does $\gamma$.

Clearly $\gamma$ can be continuously extended to $\gamma\left(1\right) = q$, and then $\gamma\inv(I^+(p)) = \left\{1\right\}$, which is not open in $[0,1]$.
\end{proof}


\section{The Time-Separation Topology}\label{sec:Separation}


In Lorentzian geometry there is a relation between causality and the time separation function $T^S$ defined in~\eqref{eq:time_separation}. To emphasise the analogy between the Lorentzian distance function and the Riemannian distance function define the outer balls
\[
O^+(p,\epsilon) = \{q\in M \:|\: T^S(p,q)>\epsilon\}, \qquad O^-(p,\epsilon) = \{q\in M\:|\:T^S(q,p)>\epsilon\}.
\]
The outer balls suggest a way to introduce a topology related to the function $T^S$.

\begin{definition}
The topology $\tau_\ts$ created by the subbasis 
\[
\mathscr{S} = \left\{O^\pm(p,\epsilon)\vert\ p\in M,\ \epsilon>0\right\}
\]
is called time separation topology. 
\end{definition}

\begin{proposition}\label{th:AlexOut_equivalence}\cite{BEE96}
Let $(M,g,T)$ be a space-time. Then $\tau_\ts = \mathscr{A} = \mathscr{A}_o$.
\end{proposition}

As a consequence of Proposition~\ref{transitiv} in space-times we have $\ts(p,q) > 0 \LRa q\in I^+(p)$. Clearly, the implication from right to left still holds in sub-space-times, but the following example shows that the converse direction fails for some sub-space-times, and that Proposition~\ref{th:AlexOut_equivalence} can not be extended.

\begin{example}\label{weirdstuff1}
Take the sub-space-time $M=\R^3$ with the bracket generating distribution generated by $X={\partial y} + x^2{\partial z}$, $Y={\partial x}$. Define a metric $g(X,X) = a$, $g(X,Y) = b$, $g(Y,Y) = 1$, where functions $a,b\in C^\infty(\R^3)$ are such that
\[a
\leq 0,\qquad a-b^2 < 0, \qquad a(0,y,0) = \begin{cases}-1&0\leq y\leq \frac{1}{3}\\ 0 &\frac{2}{3}\leq y\leq 1.\end{cases}
\]
The metric $g$ is sub-Lorentzian since the matrix $g = \begin{pmatrix}a&b\\b&1\end{pmatrix}$
has one positive and one negative eigenvalue due to $\mathrm{det}(g) = a-b^2 < 0$. The vector field $T= X-b\:Y$ will serve as time orientation. Indeed, 
\begin{align*}
g(T,X) &= g(X-bY,X)= g(X,X) - b\cdot g(X,Y)= a-b^2<0,\\
g(T,Y) &= g(X,Y) - b\:g(Y,Y)= b-b= 0,\\
\Longrightarrow\qquad g(T,T) &= g(X,T) - b\:g(Y,T)= a-b^2<0.
\end{align*}
In this sub-space-time $\ts(0,\:(0,1,0)\:) >0$ but $(0,1,0)\notin I^+(0)$. Also $0\ll (0,\frac{1}{3},0)$ and $(0,\frac{1}{3},0)\leq (0,1,0)$ but $0\not\ll (0,1,0)$.
To prove it we consider $\gamma\colon [0,1]\to M$, $\gamma(t) = (0,t,0)$. We know that $\dot\gamma(t) = X\big(\gamma(t)\big)$, and it implies $g(\dot\gamma,\dot\gamma) = a \leq 0$ and $g(\dot\gamma,T) = a-b^2 < 0$. Thus, the curve $\gamma$ is nspc.f.d.\ and has positive length $L(\gamma) \geq \frac{1}{3}$, and it is t.f.d.\ between $(0,\frac{2}{3},0)$ and $(0,1,0)$.

Like in Example~\ref{specific example1} one sees that the curve $\gamma$ is the only  f.d.\ curve connecting $0$ and $(0,1,0)$ up to monotone reparametrization. However, monotone reparametrization does not influence causal character.
\end{example}

\begin{theorem}
There are sub-space-times where $\tau_\ts \neq \mathscr{A}$.
\end{theorem}

\begin{proof}
Take the sub-space-time from Example~\ref{weirdstuff1}. Here $p=(0,1,0)\in O^+\left(0,\frac{1}{4}\right)$. However, since $a\equiv 0$ around $p$ and since the distribution is two-step generating away from $\{x=0\}$ the only possible nspc.\ Goh-curves around $p$ are null. Therefore, if $p\in I^+(q)$, it cannot lie on the boundary, so $p\in \mathrm{int}(I^+(q))$ and the same for $I^-(q)$. Hence any neighbourhood $U$ of $p$, that is open in the Alexandrov topology, contains a whole neighbourhood $V\subset U$ of $p$ in the manifold topology. But $p\in \partial I^+(0)\subset\overline{I^+(0)}$ and therefore also $p\in \partial O^+(0,\frac{1}{4})$ which means $V\not\subset O^+(0,\frac{1}{4})$.
Therefore the set $O^+(0,\frac{1}{4})$ is not open in the Alexandrov topology.
\end{proof}

Note that the sub-space-time of Example~\ref{weirdstuff1}, neither $\mathscr{A}=\tau$ nor $\tau_\ts = \tau$, as $O^+(0,\quarter)$ and $I^+(0)$ are not open in the manifold topology. In~\cite{G1} the existence of chronologically open sub-space-times where $\tau_\ts \neq \tau$ is stated. We do not know either examples, where $\mathscr{A}\not\subset \tau_\ts$, or a proof excluding this case.

\begin{theorem}
Let $(M,D,g,T)$ be a sub-space-time with manifold topology $\tau$. Then  always $\mathscr A_o\subset \tau_{\ts}$ and $\tau_\ts \subset\tau$ implies $\mathscr{A}\subset\tau$. If $\tau_\ts = \tau$, then also $\tau = \mathscr{A}$. If $\tau_\ts\subset \tau$ and $\mathscr{A}=\tau$, then also $\tau_\ts=\tau$.
\end{theorem}

\begin{proof}
Take a set $A\in\mathscr A_o$ and a point $p\in A$. Without loss of generality $A= \mathrm{int}(I^+(q))\cap\mathrm{int}(I^-(r))$ for some points $q,r\in M$. Take a t.f.d.\ curve $\gamma:[0,1]\to M$ such that $\gamma(0) = q, \gamma(\frac12)=p, \gamma(1)=r$ and define $p_1:= \gamma(\frac14), p_2:=\gamma(\frac34)$. By the transitivity of $\leq,\ll_o$ we obtain

\[p\in \quad O^+\left(p_1,2^{-1}\ts(p_1,p)\right)\cap O^-\left(p_2,2^{-1}\ts(p,p_2)\right) \quad\subset\quad J^+(p_1)\cap J^-(p_2)\quad\subset A\]
Concerning the second statement, we have $O^+(p,\epsilon)\subset\mathrm{int}(J^+(p)) = \mathrm{int}(I^+(p))$ because $O^+(p,\epsilon) \subset J^+(p)$ and $O^+(p,\epsilon)$ is open. Therefore
\[\bigcup_{\epsilon>0}O^+(p,\epsilon) \subset I^+(p).
\]
But for any $q\in I^+(p)$ the value $T^S(p,q)$ is positive or infinite, so $I^+(p)\subset\bigcup_{\epsilon>0}O^+(p,\epsilon)$.
This implies that $I^+(p) = \bigcup_{\epsilon>0}O^+(p,\epsilon)$ and $I^+(p)$ is open in the manifold topology. 

Now assume that $\tau_\ts=\mathscr{A}$. Take $p\in M$, $\epsilon>0$, and any $q\in O^+(p,\epsilon)$. Every t.p.d.\ curve $\gamma$ starting at $p$ will initially lie in $O^+(p,\epsilon)$, as the set is open. So there is a point $r\in I^-(q)\cap O^+(p,\epsilon)$. As for $s\in I^+(r)$ we find that $\ts(p,s)\geq \ts(p,r) + \ts(r,s) > \ts(p,r) >\epsilon$, we see that $s\in O^+(p,\epsilon)$. This means that $q\in I^+(r)\subset O^+(p,\epsilon)$. The same argument holds for past outer balls, so we find $\tau=\tau_\ts \subset\mathscr{A}$. Combining this with the first statement, we obtain $\mathscr{A}=\tau$.

Now assume that $\tau_\ts\subset\tau$ and $\mathscr{A}=\tau$. Take any point $p\in M$ and $q\in I^+(p)$. Clearly $\ts(p,q)>0$, possibly infinite. Choose $\epsilon = \half \ts(p,q)$, if the time separation is finite and choose any $\epsilon>0$ otherwise. Then $q\in O^+(p,\epsilon)\subset J^+(p)$ clearly. Since $O^+(p,\epsilon)$ is open, it even holds that $O^+(p,\epsilon) = \mathrm{int}(O^+(p,\epsilon))\subset \mathrm{int}(J^+(p))= \mathrm{int}(I^+(p)) = I^+(p)$, hence $q\in O^+(p,\epsilon)\subset I^+(p)$, which means that $\tau=\mathscr{A}\subset\tau_\ts$, and finally $\tau=\tau_\ts$.
\end{proof}


\end{document}